\documentclass[amsthm]{elsart}
\usepackage{amsmath}
\newtheorem{theorem}{Theorem}[section]
\newtheorem{lemma}{Lemma}[section]
\newtheorem{corollary}{Corollary}[section]

\newtheorem{remark}{Remark}[section]

\usepackage{ifpdf}
\usepackage{graphicx,amssymb,lineno}
\usepackage{subfig}
\usepackage{longtable,psfrag}
\usepackage{epstopdf}
\usepackage{float}
\usepackage{cite}
\usepackage{mathrsfs}
\usepackage{latexsym,lineno}
\usepackage{epsfig}
\usepackage{color}
\usepackage{fleqn}
\usepackage{verbatim}\usepackage{epsf}
\usepackage{amsthm}\usepackage{graphicx, float}\usepackage{graphicx}
\usepackage{amsfonts}\usepackage{amssymb}\usepackage{graphpap}
\usepackage{epic}\usepackage{curves}
\ifpdf
\usepackage[%
  pdftitle={Instructions for use of the document class
    elsart},%
  pdfauthor={Simon Pepping},%
  pdfsubject={The preprint document class elsart},%
  pdfkeywords={instructions for use, elsart, document class},%
  pdfstartview=FitH,%
  bookmarks=true,%
  bookmarksopen=true,%
  breaklinks=true,%
  colorlinks=true,%
  linkcolor=blue,anchorcolor=blue,%
  citecolor=blue,filecolor=blue,%
  menucolor=blue,pagecolor=blue,%
  urlcolor=blue]{hyperref}
\else
\usepackage[%
  breaklinks=true,%
  colorlinks=true,%
  linkcolor=blue,anchorcolor=blue,%
  citecolor=blue,filecolor=blue,%
  menucolor=blue,pagecolor=blue,%
  urlcolor=blue]{hyperref}
\fi

\makeatletter
\def\elsartstyle{%
    \def\normalsize{\@setfontsize\normalsize\@xiipt{14.5}}
    \def\small{\@setfontsize\small\@xipt{13.6}}
    \let\footnotesize=\small
    \def\large{\@setfontsize\large\@xivpt{18}}
    \def\Large{\@setfontsize\Large\@xviipt{22}}
    \skip\@mpfootins = 18\p@ \@plus 2\p@
    \normalsize
}
\@ifundefined{square}{}{}
\makeatother

\begin{document}

\begin{frontmatter}
\title{On Hermitian Adjacency Matrices for Mixed Graphs }

\journal{~~}

\author[CW]{Tao She}\ead{shetao@mails.ccnu.edu.cn},
\author[CW]{Chunxiang Wang*}\ead{wcxiang@mail.ccnu.edu.cn}

\address[CW]{ School of Mathematics and Statistics, Central China Normal University, Wuhan,  P.R. China}

\corauth[cor]{Corresponding author:  Chunxiang Wang}

\begin{abstract}
We study the spectra of mixed graphs about its Hermitian
adjacency matrix of the second kind (i.e. N-matrix) introduced by Mohar \cite{2019A}. We extend some results and  define one new Hermitian
adjacency matrix, and the entry corresponding to an arc
from u to v is equal to the k-th( or the third) root of unity, $i.e.$ $\omega=cos(\frac{2\pi}{k})+ \textbf{\emph{i}} \;sin(\frac{ 2\pi}{k})$, $k\geq 3$; the entry corresponding to an undirected edge is equal to 1, and 0 otherwise. In this paper,  we characterize  the cospectrality conditions for mixed graphs with the same underlying graph. In section 4, we determine a sharp upper bound on the spectral radius of mixed graphs,  and  provide  the corresponding extremal graphs.


\end{abstract}

\end{frontmatter}

\section{Introduction and Preliminaries}

In 2011, Reff \cite{2011Spectral} proposed the $\mathbb{T}$-gain graph (or complex unit gain graph) and its adjacency matrix.
The circle group $\mathbb{T}$ is the multiplicative group of all complex numbers with norm 1, i.e. $\mathbb{T} = \{ z \in \mathbb{C} : |z| = 1 \}$. Let $\mathbb{T}_n$ denotes the $n$-th roots of the unity, i.e. $\mathbb{T}_n = \{ z \in \mathbb{C} : z^n = 1 \}$, is a subgroup of $\mathbb{T}$. The adjacency matrix $A(M_G) = (m_{ij}) \in \mathbb{C}^{n \times n}$ of a $\mathbb{T}$-gain graph $M_G$ is defined by
\begin{eqnarray}  \label{def:1}
 m_{ij} =
\left\{
\begin{aligned}
&\varphi(e_{ij}) \quad if \; v_i \; is \; adjacent \; to \; v_j \\
&0 \quad \quad \quad otherwise
\end{aligned},
\right.
\end{eqnarray}
where $e_{ij}$ denotes the edge between $v_i$ and $v_j$, $\varphi(e_{ij})$ = $\varphi(e_{ji})^{-1}$ = $\overline{\varphi(e_{ji})} \in \mathbb{T}$.

A mixed graph is a graph with some undirected edges and some directed edges. Let $M_G$ denote a mixed graph whose underlying graph is a simple graph $G$.
$M_G$ has vertex set $V(M_G)$ and edge set $E(M_G)$. The vertices adjacent to $v \in V(M_G)$ are partitioned into three sets, that is, 
\\ $N_G(v) =N^0_{M_G}(v) \bigcup N^+_{M_G}(v) \bigcup N^-_{M_G}(v)$, where
\\$N^0_{M_G}(v) =\{u \in V(M_G): \{u, v\} \in E(M_G)\}$,
\\$N^+_{M_G}(v) =\{u \in V(M_G): \overrightarrow{vu} \in E(M_G)\}$,
\\$N^-_{M_G}(v) =\{u \in V(M_G): \overrightarrow{uv} \in E(M_G)\}$.
\\ We say that two mixed graphs are cospectral if they have the same spectrum. In this paper, we mainly consider the cospectrality of mixed graphs with the same underlying graph.

 The degree of a vertex in a mixed graph $M_G$ is defined to be the degree of this vertex in the underlying graph $G$. We say that $M_G$ is a connected mixed graph if $G$ is connected. The spectral radius of $M_G$ is defined as $\rho(M_G) = max \{ |\lambda| \}$, where $\lambda$ iterates through the eigenvalue set of $M_G$.

Mohar \cite{2019A} introduced a new Hermitian matrix ($N$-matrix for short) with $\varphi(e_{ij}) \in \mathbb{T}_6$ in (\ref{def:1}) and studied the eigenvalues of this new matrix.
Shuchao Li et al. \cite{2021Hermitian} investigate some basic properties of the $N$-matrix, which may be viewed as a continuance of Mohar's work \cite{2019A}.
They give some equivalent switchings,  which are called two-way and three-way switchings, to get cospectral graphs for a mixed graph.
Some equivalent conditions are deduced, and a sharp upper bound on the spectral radius of mixed graphs is established.

Inspired by the conclusions introduced by Mohar\cite{2019A} and Shuchao Li et al. \cite{2021Hermitian} above, we find that $\varphi(e_{ij}) \in \mathbb{T}_3$ in (\ref{def:1}) is also feasible. Furthermore, we can study all the cases of $\varphi(e_{ij}) \in \mathbb{T}_k$ ($k \geq 3$). Thus, we extend their results.
In this paper,  we characterize  the cospectrality conditions for mixed graphs with the same underlying graph. In addition,   we determine a sharp
upper bound on the spectral radius of mixed graphs,  and  provide  the corresponding extremal graphs.

\par We define a new Hermitian adjacency matrix $H(M_G) = (h_{ij})_{n \times n}$ for a mixed graph $M_G$, where
\begin{eqnarray} \label{def:2}
 h_{ij} =
\left\{
\begin{aligned}
&\omega, \quad if \; \overrightarrow{u_i u_j} \; is \; an \; arc \; from \; u_i \; to \; u_j; \\
&\overline{\omega}, \quad if \; \overrightarrow{u_j u_i} \; is \; an \; arc \; from \; u_j \; to \; u_i; \\
&1, \quad if \; u_i \; u_j \; is \; an \; undirected \; edge; \\
&0, \quad otherwise,
\end{aligned}
\right.
\end{eqnarray}
and $\omega=cos(\frac{2\pi}{k})+ \textbf{\emph{i}} \;sin(\frac{ 2\pi}{k}), k\geq 3$ is the $k$-th root of unity, $\overline{\omega}$ is the complex conjugate of $\omega$.

Let $\overline{\alpha}$ denote the the complex conjugate of $\alpha \in \mathbb{C}$. Let $\overline{H^T(M_G)}$ denote the conjugate transpose of $H(M_G)$, then $\overline{H^T(M_G)} = H(M_G)$ and all eigenvalues of $H(M_G)$ are real.

A mixed graph is called a mixed path (a mixed cycle, a mixed tree) if its underlying graph is a path (a cycle, a tree) respectively. Let $M_C = v_1 v_2 v_3 \cdots v_{t-1} v_t v_1$ be a mixed cycle. We define the weight of $M_C$ in a direction as:

$W(M_C) = h_{12} h_{23} \cdots h_{(t-1)t} h_{t1}$, \; where $h_{ij}$ is the $(v_i, v_j)$-entry of $H(M_G)$.

It is obvious that $W(M_C) \in \mathbb{T}_k$. If the weight of a mixed cycle is $\alpha$ in one direction, then the weight in the reversed direction is $\overline{\alpha}$. For a complex number $c$, let $\mathfrak{R}(c)$ denote the real part of $c$.

\section{Switching equivalence and cospectrality for $\mathbb{T}_k$ }

In this section, we focus on the sufficient conditions of cospectrality.

A corollary of Reff(\cite{Reff2016Oriented}, Lemma 2.2) is shown as below.

\begin{corollary} \label{coro:1}
Let $M_G$ and $M'_G$ be be mixed graphs with the same underlying graph $G$. If for every cycle $C$ in $G$, $\mathfrak{R}(W(M_C)) = \mathfrak{R}(W(M'_C))$, then $M_G$ and $M'_G$ are cospectral, where $M_C$(resp. $M'_C$) is a mixed cycle in $M_G$(resp. $M'_G$) whose underlying graph is $C\frac{}{}$.
\end{corollary}

Considering mixed graphs with no cycle, by Corollary \ref{coro:1}, we have the following result.
\begin{corollary} \label{coro:1.2}
A mixed tree $M_T$ is cospectral to its underlying graph $T$.
\end{corollary}
Since spectrum of some spectial trees (e.g. path, star) is known, by Corollary \ref{coro:1.2}, we can easily get the spectrum of corresponding mixed trees.

It is well known that similar matrices have the same eigenvalues.
\begin{lemma} \label{lem:1}
For any invertible matrix $P$, if $H(M'_G) = P^{-1} H(M_G) P$, then $M'_G$ and $M_G$ are cospetral.
\end{lemma}

Suppose that the vertex set of $M_G$ is partitioned into $k$ (possibly empty) sets,
\begin{equation} \label{eq:2.1}
  V(M_G) = V_1 \cup V_2 \cup \cdots \cup V_k.
\end{equation}

An arc $\overrightarrow{xy}$ or an undirected edge $\{x, y\}$ is said to be of type $(\omega ^ i, \omega ^ j)$ if $x \in V_i$ and $y \in V_j$, where $i, j \in  \{ 1, 2, \cdots, k \}$, $\omega ^ i \in \mathbb{T}_k$,  $\omega ^ j  \in \mathbb{T}_k$. The partition (\ref{eq:2.1}) is said to be \emph{admissible} if both of the following two conditions hold (refer to Fig. \ref{fig:threeWaySwith_Tk}(a)):
\\ (i) each undirected edge is of the type $(\omega ^ i, \; \omega ^ j)$, where $j - i \equiv 0$ or $1$ or $-1$ (mod $k$);
\\ (ii) each arc is of the type $(\omega ^ i, \; \omega ^ j)$, where $i - j \equiv 0$ or $1$ or $2$ (mod $k$).

A mixed graph is \emph{admissible} if and only if it has an admissible partition.

\begin{figure}[H]
\centering
\includegraphics[width=0.8\linewidth]{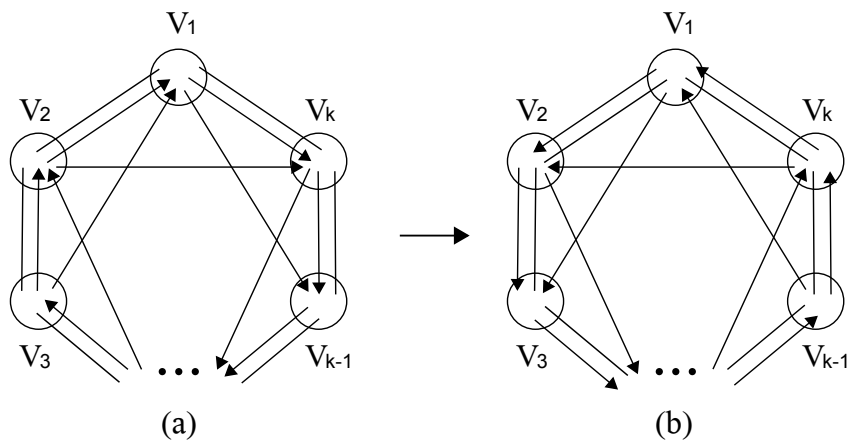}
\caption{ Three way switching for $k$ partitions with $\omega=cos(\frac{2\pi}{k})+ \textbf{\emph{i}} \;sin(\frac{ 2\pi}{k}) \in \mathbb{T}_k$. }
\label{fig:threeWaySwith_Tk}
\end{figure}

\textbf{Three-way switching for $\mathbb{T}_k$:}
\\ A \emph{three-way switching} for $\mathbb{T}_k$ with respect to the admissible partition (\ref{eq:2.1}) is the operation of changing $M_G$ into the mixed graph $M_G'$ by making the changes in what follows (see Fig.\ref{fig:threeWaySwith_Tk}):
\\ (i) replacing each undirected edge of type $(\omega ^ i, \; \omega ^ j)$ with an arc of type $(\omega ^ i, \; \omega ^ j)$, where $j - i \equiv 1$ (mod $k$);
\\ (ii) replacing each arc of type $(\omega ^ i, \; \omega ^ j)$ with an undirected edge, where $i - j \equiv 1$ (mod $k$);
\\ (iii) reversing the direction of each arc of type $(\omega ^ i, \; \omega ^ j)$, where $i - j \equiv 2$ (mod $k$).

Let $n_i = \left| V_i \right|$, where $i \in \{ 1, 2, \cdots, k \}$, then $\sum \limits_{i=1}^{k} n_i = \left| V(M_G) \right| = n$.
Let $D_i = d_i I_{n_i} = diag\{ d_i, d_i, \cdots, d_i \}$  be a $n_i * n_i$ diagonal matrix, where $d_i \neq 0$ and $I_{n_i}$ is a $n_i * n_i$ identity matrix. Let
\begin{eqnarray} \label{matrix:G1}
D \;\;=\;\;
\begin{bmatrix}
D_1 \quad & \quad & \quad & \quad  \\
\quad & D_2 \quad & \quad & \quad  \\
\quad & \quad & \ddots \quad & \quad  \\
\quad & \quad & \quad & D_k \quad
\end{bmatrix} \; .
\end{eqnarray}

Naturally, by Lemma \ref{lem:1}, we get the following corollary.
\begin{corollary} \label{coro:2}
 $M_G$ and $M'_G$ are cospectral if $ H(M'_G) = D^{-1} H(M_G) D$, where $D$ is defined as equation (\ref{matrix:G1}).
\end{corollary}

The entries of the matrix $H(M'_G)$ are given by $h'_{uv}:=d^{-1}_i h_{uv} d_j$, where $u \in V_i, v\in V_j$.

 By equation (\ref{def:2}), we know that an entry of an Hermitian adjacency matrix belongs to $\{ \omega, \overline{\omega}, 1, 0 \}$, i.e. $h_{uv},h'_{uv} \in \{ \omega, \overline{\omega}, 1, 0 \}$. Since $d^{-1}_i \neq 0$ and $d_j \neq 0$, therefore $h'_{uv} = 0$ if and only if $h_{uv} = 0$.

By Corollary \ref{coro:2}, we can obtain some sufficient conditions of cospectrality for two mixed graph with the same underlying graph. Let $d_i = \omega ^ {i - 1}$ for $i \in \{ 1, 2, \cdots, k \}$, in view of $h_{uv},h'_{uv} \in \{ \omega, \overline{\omega}, 1, 0 \}$, we try to find a cospectral transformation of changing $M_G$ into the mixed graph $M_G'$.

Recall that the three-way switching is only for the admissible mixed graphs.
\begin{theorem}\label{th:3.1}
If mixed graph $M_G'$ is obtained from mixed graph $M_G$ by the three-way switching, then $M_G'$ is cospectral with $M_G$.
\end{theorem}
\begin{proof}
Let $D$ be a diagonal matrix as shown in equation (\ref{matrix:G1}) and let $d_i = \omega ^ {i - 1}$.
Let $H(M_G)=(h_{ij})_{n \times n}$, $H(M'_G) = (h'_{ij})_{n \times n} = D^{-1} H(M_G) D$.
Let $V(M_G) = V_1 \cup V_2 \cup \cdots \cup V_k$ be  an admissible partition of $M_G$. Let $u \in V_i$, $v\in V_j$.
Recall that a three-way switching contains operations (i), (ii), (iii).
\\ For operation (i), suppose $uv$  is an  undirected edge in $M_G$, where $j - i \equiv 1$ (mod $k$).
By equation (\ref{def:2}), we have $h_{uv} = h_{vu} = 1$. Then $h'_{uv}:=d^{-1}_i h_{uv} d_j$ = $\omega  h_{uv}$ = $\omega $, $h'_{vu}:=d^{-1}_j h_{vu} d_i$ = $\omega ^{-1} h_{vu}$ = $\omega ^{-1}$.
That is,  $uv$ is an arc of type $(\omega ^ i, \; \omega ^ j)$ in $M'_G$.
\\ For operation (ii), suppose $uv$ is an arc of type $(\omega ^ i, \; \omega ^ j)$, where $i - j \equiv 1$ (mod $k$).
By equation (\ref{def:2}), we have $h_{uv} = \omega$ and $h_{vu} = \overline{\omega}$. Then $h'_{uv}:=d^{-1}_i h_{uv} d_j$ = $\omega ^{-1} h_{uv}$ = $1$, $h'_{vu}:=d^{-1}_j h_{vu} d_i$ = $\omega  h_{vu}$ = $1$.
That is,  $uv$ is an undirected edge in $M'_G$.
\\ For operation (iii), suppose $uv$ is an arc of type $(\omega ^ i, \; \omega ^ {j})$ in $M_G$, where $i - j \equiv 2$ (mod $k$).
By equation (\ref{def:2}), we have $h_{uv} = \omega$ and $h_{vu} = \overline{\omega}$. Then $h'_{uv}:=d^{-1}_i h_{uv} d_j$ = $\omega ^{-2} h_{uv}$ = $\omega ^{-1}$, $h'_{vu}:=d^{-1}_j h_{vu} d_i$ = $\omega ^{2} h_{vu}$ = $\omega$.
That is,  $uv$ is an arc of type $(\omega ^ j, \; \omega ^ i)$ in $M'_G$.

Therefore,  for a mixed graph which is admissible, the  three-way switching is
equivalent to the similar transformation $H(M'_G) = D^{-1} H(M_G) D$. By Corollary \ref{coro:2}, we infer that $M'_G$ is cospectral with $M_G$.
\end{proof}

By the proof of Theorem \ref{th:3.1}, we get the following corollary \ref{coro:2.0}.
\begin{corollary} \label{coro:2.0}
For a mixed graph $M_G$ which is admissible, a three-way switching transforming $M_G$ into $M_G'$
corresponds to a similar transformation $H(M'_G) = D^{-1} H(M_G) D$, where $D$ defined as equation (\ref{matrix:G1}).
\end{corollary}

\textbf{Two-way switching for $\mathbb{T}_k$:}
\\ In particular, suppose that the vertex set of $M_G$ is partitioned into $2$ sets, i.e.
$
  V(M_G) = V_1 \cup V_2
$
, then equation (\ref{matrix:G1}) becomes
$
\begin{bmatrix}
D_1 \quad & \quad   \\
\quad & D_2
\end{bmatrix} \;
$.
If $M_G$ contains no arc $\overrightarrow{u_2 u_1}$ with $u_1 \in V_1$ and $u_2 \in V_2$, then a \emph{two-way switching} for $\mathbb{T}_k$ is said to be the operation of making two kinds of changes as follow (see Fig.\ref{fig:twoWaySwith_Tk}):
\\ (i) replaces every arc $\overrightarrow{u_1 u_2}$ with an undirected edge $\{u_1, u_2\}$, where $u_1 \in V_1$ and $u_2 \in V_2$;
\\ (ii) replaces every undirected edge $\{u_1, u_2\}$ with an arc $\overrightarrow{u_2 u_1}$, where $u_1 \in V_1$ and $u_2 \in V_2$.

\begin{figure}[H]
\centering
\includegraphics[width=0.6\linewidth]{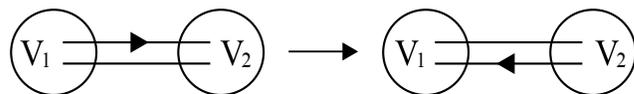}
\caption{ Two-way switching for $2$ partitions with $\omega=cos(\frac{2\pi}{k})+ \textbf{\emph{i}} \;sin(\frac{ 2\pi}{k}) \in \mathbb{T}_k$. }
\label{fig:twoWaySwith_Tk}
\end{figure}

\begin{corollary}\label{th:3.2}
The mixed graph $M_G'$ obtained from $M_G$ by the two-way switching is cospectral with $M_G$.
\end{corollary}

By setting $k=6$ and $\omega=cos(\frac{2\pi}{6})+ \textbf{\emph{i}} \;sin(\frac{ 2\pi}{6})$ in theorem \ref{th:3.2}, \ref{th:3.1}, we can get the two-way switching and three-way switching results in \cite{2021Hermitian}.

\par \quad The converse of $M_G$, denoted by $M_G^c$, is obtained by reversing all the arcs of $M_G$.

Given a mixed graph $M_G$, let $M_G^c$ be its converse(the mixed graph obtained by reversing all the arcs of $M_G$). By (\ref{def:2}), we get $H(M_G^c) = H(M_G)^T$ , this implies the following result.

\begin{theorem}\label{th:4.3}
A mixed graph $M_G$ and its converse $M_G^c$ are cospectral.
\end{theorem}

\par \quad If a mixed graph $M_G$ can be turned into another mixed graph $M'_G$ after a sequence of three-way switchings and operations of taking the converse, then we say $M_G$ and $M'_G$ are switching equivalent.

\par \quad Theorem 4.1 of \cite{Bojan2016Hermitian} characterizes the equivalent conditions for the mixed graph being cospectral to its underlying graph, for $\omega = cos(\frac{2\pi}{4})+ \textbf{\emph{i}} \;sin(\frac{ 2\pi}{4}) = i$ in (\ref{def:2}). We can get a similar result for $\omega=cos(\frac{2\pi}{k})+ \textbf{\emph{i}} \;sin(\frac{ 2\pi}{k}), k\geq 3$  as follow.

\begin{theorem}\label{th:4.4}
Let $G$ be a connected simple graph of order $n$ and let $M_{G_1}$ be a mixed graph whose underlying graph $G_1$ is a spanning subgraph of $G$. Then the following statements are equivalent:
\\ (a) $G$ and $M_{G_1}$ are cospectral.
\\ (b) $\lambda_1(G) = \lambda_1(M_{G_1})$.
\\ (c) $G_1 = G $, and the vertex set of $M_{G_1}$ has a partition $V_1 \cup V_2 \cup \cdots \cup V_k$ such that the following holds: every undirected edge $uv$ of $M_{G_1}$ satisfies $u \in V_j$ and $v \in V_j$ for some $j \in \{ 1, 2, \cdots, k \}$; and every arc $\overrightarrow{uv}$ of $M_{G_1}$ satisfies $u \in V_r$ and $v \in V_s$ for some $r,s \in \{ 1, 2, \cdots, k \}$, where $r - s \equiv 1$ (mod $k$).
\\ (d) $G$ and $M_{G_1}$ are switching equivalent.
\end{theorem}
\begin{proof}
It is obvious that (a) implies (b) and (d) implies (a). By theorem \ref{th:3.1}, (c) implies (d). We just need to prove that (b) implies (c).
Assume that (b) holds. Let $H(G_1)=(a_{ij})_{n \times n}$, $H(M_{G_1}) = (h_{ij})_{n \times n}$. Let $\textbf{x} =(x_1, ..., x_n)^T \in C^n$ be a normalized eigenvector of $H(G_1)$ corresponding to $\lambda_1(G_1)$, let $\textbf{y} =(y_1, ..., y_n)^T \in C^n$ be a normalized eigenvector of $H(M_{G_1})$ corresponding to $\lambda_1(M_{G_1})$. We have
\begin{eqnarray}
\nonumber  \lambda_1(M_{G_1}) &=& \overline{\textbf{y}^T} \cdot H(M_{G_1}) \cdot \textbf{y} \\
\nonumber   &=& \sum \limits_{ u \in V}^{} \sum \limits_{ v \in V}^{} \overline{y_u} h_{uv} y_v \\
\label{eq:4.1.1}   &\leq & \sum \limits_{ u \in V}^{} \sum \limits_{ v \in V}^{} \left|  \overline{y_u} h_{uv} y_v \right| \\
\nonumber   &= & \sum \limits_{ u \in V}^{} \sum \limits_{ v \in V}^{} \left|  \overline{y_u}  y_v \right| \cdot \left| h_{uv} \right| \\
\nonumber   &= & \sum \limits_{ u \in V}^{} \sum \limits_{ v \in V}^{} \left|  \overline{y_u}  y_v \right| \cdot a_{uv} \qquad ( \left| h_{uv} \right| = a_{uv}, \; since \;  h_{uv}  \in \{ \omega, \overline{\omega}, 1, 0 \} ) \\
\label{eq:4.1.2}    &\leq & \sum \limits_{ u \in V}^{} \sum \limits_{ v \in V}^{}  \overline{x_u}  x_v  \cdot a_{uv} = \lambda_1(G_1)
\end{eqnarray}

Since $G_1$ is a spanning subgraph of $G$, we get $\lambda_1(G_1) \leq \lambda_1(G)$. Combined with (b), i.e. $\lambda_1(G) = \lambda_1(M_{G_1})$, we  get $\lambda_1(G) = \lambda_1(M_{G_1}) \leq \lambda_1(G_1) \leq \lambda_1(G)$. Hence
\begin{equation}\label{eq:4.2}
  \lambda_1(M_{G_1}) = \lambda_1(G) = \lambda_1(G_1).
\end{equation}
It is well known that $\lambda_1(G_1) < \lambda_1(G)$ if $G_1$ is a proper subgraph of a connected graph $G$. Therefore, $G_1=G$.

The equality in (\ref{eq:4.1.1}) holds requires
\begin{equation}\label{eq:4.3}
  \overline{y_u} h_{uv} y_v = \left|  \overline{y_u} h_{uv} y_v \right|
\end{equation}
for every edge $uv$. Without loss of generality, we can assume that $y_1 \in \mathbb{R}^+$ for $\textbf{y} \neq 0$, then $y_1/|y_1| = 1$.
Considering equation (\ref{eq:4.3}), we infer that if $v_i \in N^0_{M_{G_1}}(v_1)$, then $h_{1i}=1$, $y_i/|y_i| = 1$; if $v_i \in N^+_{M_{G_1}}(v_1)$, then $h_{1i}=\omega$, $y_i/|y_i| = \overline{\omega}$; if $v_i \in N^-_{M_{G_1}}(v_1)$, then $h_{1i}=\overline{\omega}$, $y_i/|y_i| = \omega$.
\par \quad Note that $G_1$ is connected. By repeating the above steps, we have $y_i/|y_i| \in \mathbb{T}_k$ for $i \in \{1, \cdots, n \}$.
Let $V_j= \{v_i \in V(M_G) :y_i/|y_i| = \omega ^ j\}$, $j \in \{ 1, 2, \cdots, k \}$. Then they form a partition of $V(M_G)$, and satisfy the requirements of (c).
\end{proof}

\section{Switching equivalence and cospectrality for $\mathbb{T}_3$ }

For $\mathbb{T}_3$, we have more simple results than for $\mathbb{T}_k$ .

\textbf{Three-way switching for $\mathbb{T}_3$:}
\\ Suppose that the vertex set of $M_G$ is partitioned into $3$ (possibly empty) sets, i.e.
$
  V(M_G) = V_1 \cup V_2 \cup V_3
$. The \emph{three-way switching} for $\mathbb{T}_3$ is as follows (see Fig.\ref{fig:threeWaySwith}):
\\ (i) replacing each undirected edge of type $(\omega ^ i, \; \omega ^ j)$ with an arc of type $(\omega ^ i, \; \omega ^ j)$, where $j - i \equiv 1$ (mod $3$);
\\ (ii) replacing each arc of type $(\omega ^ i, \; \omega ^ j)$ with an undirected edge, where $i - j \equiv 1$ (mod $3$);
\\ (iii) reversing the direction of each arc of type $(\omega ^ i, \; \omega ^ j)$, where $i - j \equiv 2$ (mod $3$).

\begin{figure}[H]
\centering
\includegraphics[width=0.6\linewidth]{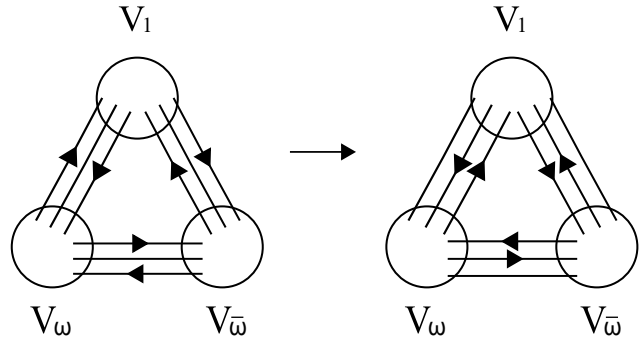}
\caption{ Three-way switching for three partitions with $\omega=-\frac{1}{2}+ \textbf{\emph{i}} \;\frac{\sqrt{3}}{2} \in \mathbb{T}_3$. }
\label{fig:threeWaySwith}
\end{figure}

\par \qquad Suppose that the vertex set of $M_G$ is partitioned into $2$ sets, i.e.
$
  V(M_G) = V_1 \cup V_2
$. Then the three-way switching for $\mathbb{T}_3$ is shown as  Fig.\ref{fig:twoWaySwith}.

\begin{figure}[H]
\centering
\includegraphics[width=0.6\linewidth]{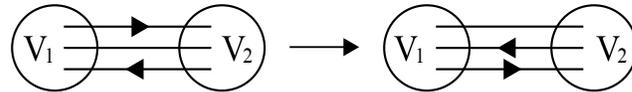}
\caption{ Three-way switching for two partitions with $\omega=-\frac{1}{2}+ \textbf{\emph{i}} \;\frac{\sqrt{3}}{2} \in \mathbb{T}_3$. }
\label{fig:twoWaySwith}
\end{figure}

By Theorem \ref{th:3.1}, we can get the following corollary immediately.
\begin{corollary}\label{coro:3.1}
Let $M_G$ be a mixed graph with $\omega=-\frac{1}{2}+ \textbf{\emph{i}} \;\frac{\sqrt{3}}{2} \in \mathbb{T}_3$. If mixed graph $M_G'$ is obtained from $M_G$ by the three-way switching, then $M_G'$ is cospectral with $M_G$.
\end{corollary}

\begin{remark}
Notice that the cospectrality of two fixed graph depends on the value of $\omega$. For example, let $C'_6$ denotes a $6$-vertex mixed graph whose underlying graph is $C_6$ and all edges are arcs in one direction.
By the above results, we have that $C_6$ and $C'_6$ are cospectral when $\omega=cos(\frac{2\pi}{3})+ \textbf{\emph{i}} \;sin(\frac{ 2\pi}{3})$ or $\omega=cos(\frac{2\pi}{6})+ \textbf{\emph{i}} \;sin(\frac{ 2\pi}{6})$, but not cospectral when $\omega=cos(\frac{2\pi}{4})+ \textbf{\emph{i}} \;sin(\frac{ 2\pi}{4})$ or $\omega=cos(\frac{2\pi}{5})+ \textbf{\emph{i}} \;sin(\frac{ 2\pi}{5})$.
\end{remark}

\section{An upper bound for the spectral radius}

For a mixed graph $M_G$ with $\omega=\frac{1}{2}+ \textbf{\emph{i}} \;\frac{\sqrt{3}}{2} \in \mathbb{T}_6$, Theorem 3.1 of \cite{2021Hermitian} shows that  $\rho (M_G) \leq \Delta(G)$, where $\rho (M_G)$ and $\Delta(G)$  denote the spectral radius and the maximum degree of $M_G$, respectively.
And more generally, for $M_G$ with $\omega=cos(\frac{2\pi}{k})+ \textbf{\emph{i}} \;sin(\frac{ 2\pi}{k}), k\geq 3$, we also get that $\rho (M_G) \leq \Delta(G)$ and the
 extremal graphs can be partitioned as the form in Fig. \ref{fig:mixedGraph_Tk}(a).

\begin{figure}[H]
\centering
\includegraphics[width=0.6\linewidth]{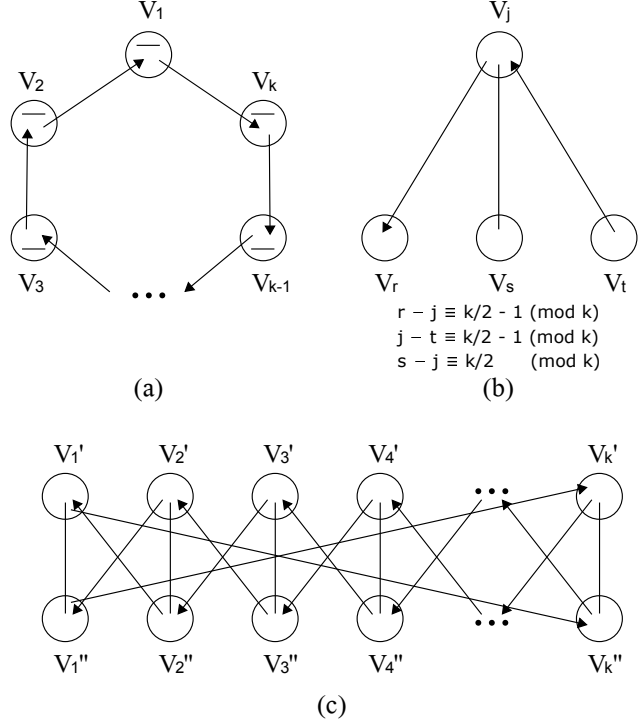}
\caption{ Partition for Theorem \ref{th:2.1} }
\label{fig:mixedGraph_Tk}
\end{figure}

\begin{theorem}\label{th:2.1}
Let $M_G$ be a connected mixed graph with $\omega=cos(\frac{2\pi}{k})+ \textbf{\emph{i}} \;sin(\frac{ 2\pi}{k}), k\geq 3$. Let $\lambda$ be an arbitrary eigenvalue of $H(M_G)$,
then  $ |\lambda| \leq \Delta(G) $.
Let $\rho(M_G) = max \{ |\lambda| \}$ denote the spectral radius of $M_G$.
Then $\rho(M_G) = \Delta(G)$ if and only if $G$ is $\Delta(G)$-regular and one can partition $V(M_G)$ into $k$ (possibly empty) parts $V_{1}, V_{2}, \cdots, V_{k}$ such that one of the followings holds:

(i)
Every undirected edge $uv$ of $M_G$ satisfies $u \in V_j$ and $v \in V_j$ for some $j \in \{ 1, 2, \cdots, k \}$; and every arc $\overrightarrow{uv}$ of $M_G$ satisfies $u \in V_r$ and $v \in V_s$ for some $r,s \in \{ 1, 2, \cdots, k \}$, where $r - s \equiv 1$ (mod $k$). (see Fig. \ref{fig:mixedGraph_Tk}(a).)

(ii)
$k$ is even. The induced mixed graph $M_G[V_j]$ is an independent set for all $j \in \{ 1, 2, \cdots, k \}$; every undirected edge $uv$ of $M_G$ satisfies $u \in V_j$ and $v \in V_{s}$ for some $j,s \in \{ 1, 2, \cdots, k \}$, where $s - j \equiv \frac{k}{2}$ (mod $k$); and every arc $\overrightarrow{uv}$ of $M_G$ satisfies $u \in V_j$ and $v \in V_r$ for some $j,r \in \{ 1, 2, \cdots, k \}$, where $r - j \equiv \frac{k}{2} - 1$ (mod $k$). (Fig. \ref{fig:mixedGraph_Tk}(b) shows the edges joining the set $V_j$.)

(iii)
$k$ is odd. $V(M_G)$ can be partitioned into $2k$ (possibly empty) parts $V'_{1}, V'_{2}, \cdots, V'_{k}, V''_{1}, V''_{2}, \cdots, V''_{k}$, such that: every undirected edge $uv$ of $M_G$ satisfies $u \in V'_j$ and $v \in V''_j$ for some $j \in \{ 1, 2, \cdots, k \}$; and every arc $\overrightarrow{uv}$ of $M_G$ satisfies $u \in V'_r$, $v \in V''_s$ (or $u \in V''_r$, $v \in V'_s$) for some $r,s \in \{ 1, 2, \cdots, k \}$, where $r - s \equiv 1$ (mod $k$). (see Fig. \ref{fig:mixedGraph_Tk}(c))
\end{theorem}
\begin{proof}
Let $H = H(M_G) = (h_{ij})_{n \times n}$ and let $\textbf{y} = (y_1, y_2, \cdots ,y_n)^T$ be an eigenvector corresponding to the eigenvalue $\lambda$ of $H$. Without loss of generality, let $\left| y_1 \right| = max\{ \left| y_i \right| : 1 \leq i \leq n \}$. Let $\textbf{x} = (x_1, x_2, \cdots ,x_n)^T = \textbf{y} / y_1$, then $\textbf{x}$ is also an eigenvector corresponding to the eigenvalue $\lambda$ of $H$ with $x_1 = 1 = max\{ \left| x_i \right| : 1 \leq i \leq n \}$.
Let vertex $v_i$ correspond to $x_i$, for $i \in \{ 1, 2, \cdots, n \}$.  Since $H \textbf{x} = \lambda \textbf{x} $, we have
$
  (H \textbf{x}) _1 = \lambda x_1
$,
subsequently,
\begin{eqnarray}
\nonumber  \left| \lambda \right|  &&= \left| \lambda x_1 \right| = \left| (H \textbf{x})_1 \right| \\
\nonumber   &&= \left| \big( \sum \limits_{ v_s \in N_{M_G}^0(v_1)}^{} x_s \big) + \omega \Big( \sum \limits_{ v_r \in N_{M_G}^+(v_1)}^{} x_r \big) + \overline{\omega} \big( \sum \limits_{ v_t \in N_{M_G}^-(v_1)}^{} x_t \big) \right| \\
\label{eq:2.2}   &&\leq \big( \sum \limits_{ v_s \in N_{M_G}^0(v_1)}^{} |x_s| \big) + |\omega| \big( \sum \limits_{ v_r \in N_{M_G}^+(v_1)}^{} |x_r| \big) + |\overline{\omega}| \big( \sum \limits_{ v_t \in N_{M_G}^-(v_1)}^{} |x_t| \big) \\
\label{eq:2.3}   &&\leq d_G(v_1) |x_1| \\
\label{eq:2.4}   &&\leq \Delta(G) |x_1| \\
\nonumber   &&= \Delta(G).
\end{eqnarray}
Hence, $|\lambda| \leq \Delta(G) $. (Note that $\lambda$ is an arbitrary eigenvalue of $M_G$.)

In what follows, we try to find out the extremal graphs with $|\lambda| = \rho(M_G) = \Delta(G)$, which require all  the equalities in (\ref{eq:2.2}), (\ref{eq:2.3}), (\ref{eq:2.4}) hold. We see that equality in (\ref{eq:2.4}) holds if and only if $d_G(v_1) = \Delta (G)$, whereas equality in (\ref{eq:2.3}) holds if and only if
\begin{equation}\label{eq:2.5}
  |x_i| = |x_1| = 1 \quad for \; all \quad v_i \in N_G(v_1).
\end{equation}

For any $v_i \in N_G(v_1)$, we can make a similar analysis and get that $d_G(v_i) = \Delta (G)$ and $|x_j| = |x_i| = 1$ for all $v_j \in N_G(v_i)$.
Since $M_G$ is connected, the rest vertices may be deduced by analogy. Therefore $M_G$ is $\Delta (G)$ regular and
\begin{equation} \label{eq:2.5.2}
|x_i| = |x_1| = 1 \;\; for \; \;  i \in \{ 1, 2, \cdots, n \}.
\end{equation}
The equality in (\ref{eq:2.2}) holds if and only if every complex number in the following set $S$ has the same argument, where
\begin{eqnarray}
\nonumber  S = \big\{ x_s: v_s \in N_{M_G}^0(v_1) \big\} &&\cup \;\; \big\{ \omega \cdot x_r: v_r \in N_{M_G}^+(v_1) \big\}
\\ \label{eq:2.6} &&\cup \;\; \big\{ \overline{\omega} \cdot x_t: v_t \in N_{M_G}^-(v_1) \big\}.
\end{eqnarray}
While
\begin{eqnarray}
\nonumber  \lambda x_1 &&= (H \textbf{x})_1
\\ \label{eq:2.7} &&= \big( \sum \limits_{ v_s \in N_{M_G}^0(v_1)}^{} x_s \big) + \omega \Big( \sum \limits_{ v_r \in N_{M_G}^+(v_1)}^{} x_r \big) + \overline{\omega} \big( \sum \limits_{ v_t \in N_{M_G}^-(v_1)}^{} x_t \big) ,
\end{eqnarray}
then  the equality in (\ref{eq:2.2}) holds if and only if every complex number in the set $S$ has the same argument  as $\lambda x_1$.

It is well known that the trace of a square matrix equals the sum of the eigenvalues, thus the sum of the eigenvalues of $H(M_G)$ is zero. Since a mixed graph whose eigenvalues are all zero is an empty graph, it is  sufficient to consider the cases $\lambda > 0$ and $\lambda < 0$. (Note that $\lambda$ is real.)
\par \quad \textbf{Case 1.} $\lambda  = \Delta(G) > 0$. By (\ref{eq:2.6}) and (\ref{eq:2.7}), every complex number in $S$ has the same argument with $x_1 = 1$. Combined with (\ref{eq:2.5.2}), we conclude that
\begin{eqnarray*}
 x_{i} =
\left\{
\begin{aligned}
&1, \quad &if \; v_i \in N_{M_G}^0(v_1); \\
&\overline{\omega}, \quad &if \; v_i \in N_{M_G}^+(v_1); \\
&\omega, \quad &if \; v_i \in N_{M_G}^-(v_1) .
\end{aligned}
\right.
\end{eqnarray*}

Similar discussion can be applied to each $v_i \in N(v_1)$, we obtain $x_i \in \mathbb{T}_k$. And so on, we get
$x_s \in \mathbb{T}_k$ for $ s \in \{ 1, 2, \cdots, n \}$ since $M_G$ is connected. Hence, $V(M_G)$ is partitioned into
\begin{equation*}
  V_1 \cup V_2 \cup \cdots \cup V_k, \; where \; V_j = \{ v_s | x_s = \omega ^ j \in \mathbb{T}_k  \}, \; j \in \{ 1, 2, \cdots, k \}.
\end{equation*}
So (i) holds.

\par \quad \textbf{Case 2.} $\lambda  = -\Delta(G) < 0$. By (\ref{eq:2.5.2}),(\ref{eq:2.6}) and (\ref{eq:2.7}), we obtain that every complex number in $S$ equals $-x_1 = -1$. With the same discussion as that of Case 1, we conclude that
\begin{eqnarray*}
 x_{i} =
\left\{
\begin{aligned}
&-1, \quad &if \; v_i \in N_{M_G}^0(v_1); \\
&-\overline{\omega}, \quad &if \; v_i \in N_{M_G}^+(v_1); \\
&-\omega, \quad &if \; v_i \in N_{M_G}^-(v_1) .
\end{aligned}
\right.
\end{eqnarray*}
Note that $x_1 = 1 \in \mathbb{T}_k$ and $-x_i \in  \{ 1, \omega, \overline{\omega} \} \subseteq  \mathbb{T}_k$.
\\ Considering $v_i \in N_G(v_1)$, process similarly as above, we conclude that
\begin{eqnarray*}
 x_{j} =
\left\{
\begin{aligned}
&(-1)x_i, \quad &if \; v_j \in N_{M_G}^0(v_i); \\
&(-\overline{\omega})x_i, \quad &if \; v_j \in N_{M_G}^+(v_i); \\
&(-\omega)x_i, \quad &if \; v_j \in N_{M_G}^-(v_i) .
\end{aligned}
\right.
\end{eqnarray*}
We have $x_j \in \mathbb{T}_k$ for any $v_j \in N_G(v_i)$.
Since $M_G$ is connected, the rest vertices may be deduced similarly. Thus each vertex $v_s$ satisfies $x_s \in \mathbb{T}_k$ or $-x_s \in \mathbb{T}_k$, where $ s \in \{ 1, 2, \cdots, n \}$.

If $k$ is even, then $-1 = \omega ^ {k/2} \in \mathbb{T}_k$. Similar as Case 1, we have $x_i \in \mathbb{T}_k$ for $i \in \{ 1, 2, \cdots, n \}$. Thus $V(M_G)$ can be partitioned into
$
  V_1 \cup V_2 \cup \cdots \cup V_k
$  according to the value of $x_i$, and (ii) holds.

If $k$ is odd, then $-1 \notin \mathbb{T}_k$ since $ \omega ^ j \neq -1 $ for each $j \in \{ 1, 2, \cdots, k \}$. We can separate $V(M_G)$ into two independent vertex sets $I_1$ and $I_2$ according to $x_i \in \mathbb{T}_k$ or $x_i \notin \mathbb{T}_k$, where $i \in \{ 1, 2, \cdots, n \}$. That means $M_G$ is a bipartite mixed graph. Recall that $M_G$ is $\Delta (G)$ regular, so $\left| I_1 \right| = \left| I_2 \right|$.
Since $x_i \in \mathbb{T}_k$ or $-x_i \in \mathbb{T}_k$ for $i \in \{ 1, 2, \cdots, n \}$, then $V(M_G)$ can be partitioned into $V'_{1} \cup V'_{2} \cup \cdots \cup V'_{k} \cup V''_{1} \cup V''_{2} \cup \cdots \cup V''_{k}$  according to the value of $x_i$, where $ V'_j = \{ v_i | x_i = \omega ^ j \in \mathbb{T}_k  \}, \; V''_j = \{ v_i | -x_i = \omega ^ j \in \mathbb{T}_k  \}, \; j \in \{ 1, 2, \cdots, k \}$. Hence (iii) holds.

\par \qquad Now, we consider the converse for (i) and (ii) of the theorem \ref{th:2.1}. Let $M_G$ be a mixed graph whose underlying graph is $\Delta(G)$-regular. Assume that $V(M_G)$ has a partition $\bigcup _{j \in T_k}V_j$ satisfying (i) or (ii).

\par \qquad Let $\textbf{x} = (x_1, x_2, \cdots , x_n)^T$ be the vector indexed by the vertices of $M_G$ such that $x_i = \omega ^ j$ if $v_i \in V_j$, where $i \in \{ 1, 2, \cdots, n \}$, $j \in \{ 1, 2, \cdots, k \}$.
Then
\begin{eqnarray*}
  (H \textbf{x})_i &= &\big( \sum \limits_{ v_s \in N_{M_G}^0(v_i)}^{} x_s \big) + \omega \Big( \sum \limits_{ v_r \in N_{M_G}^+(v_i)}^{} x_r \big) + \overline{\omega} \big( \sum \limits_{ v_t \in N_{M_G}^-(v_i)}^{} x_t \big)
 \\ &= &
\left\{
\begin{aligned}
&\Delta(G) x_i, \quad &if \;  (i) \; of \; Theorem \; \ref{th:2.1} \; holds; \; \\
&- \Delta(G) x_i, \quad &if \;  (ii) \; of \; Theorem \; \ref{th:2.1} \; holds.
\end{aligned}
\right.
\end{eqnarray*}
Thus $\textbf{x}$ is an eigenvector of $H$ corresponding to the eigenvalue $\Delta(G)$ or $-\Delta(G)$, and so $ \rho(M_G) = max \{ |\lambda| \} = \Delta (G)$. Then the bound is tight as claimed. The proof of the converse for (iii) is omitted, since it is same as above.
\end{proof}

In fact, (iii) is a subcase of (i) by setting $V_i = V'_i \cup V''_i$ in Theorem \ref{th:2.1}. Then there are two eigenvalues $\Delta (G)$ and $-\Delta (G)$ in (iii), and the following corollary holds.
\begin{corollary}\label{coro:2.1}
Let $M_G$ be a connected mixed graph with $\omega=cos(\frac{2\pi}{k})+ \textbf{\emph{i}} \;sin(\frac{ 2\pi}{k}), k\geq 3$ and $k$ is odd. If $-\Delta(G)$ is an eigenvalue of $H(M_G)$,
then  $\Delta(G)$ is also an eigenvalue of $H(M_G)$.
\end{corollary}

\vskip4mm\noindent{\bf Conflict of interest statement }

No potential conflict of interest was reported by the authors.

\vskip4mm\noindent{\bf Acknowledgements}

 The work was partially supported by the National Natural Science Foundation of China under Grants 11771172,12061039.

\vskip4mm\noindent{\bf Data Availability Statements.}

Data sharing not applicable to this article as no datasets were generated or analysed during the current study.


\begin{thebibliography}{99}

\bibitem{2019A}
B.~Mohar.
\newblock
  \href{https://www.sciencedirect.com/science/article/pii/S0024379519304136}{A
  new kind of Hermitian matrices for digraphs}.
\newblock {\em Linear Algebra and its Applications}, 584, 2019.

\bibitem{2021Hermitian}
S.~Li and Y.~Yu.
\newblock
  \href{https://www.sciencedirect.com/science/article/pii/S0012365X22000048}{Hermitian
  adjacency matrix of the second kind for mixed graphs}.
\newblock {\em Discrete Mathematics}, 345, 2022.

\bibitem{2011Spectral}
N.~Reff.
\newblock
  \href{https://www.sciencedirect.com/science/article/pii/S002437951100718X}{Spectral
  Properties of Complex Unit Gain Graphs}.
\newblock {\em Linear Algebra and its Applications}, 436(9):3165--3176, 2011.

\bibitem{2015Hermitian}
Jianxi~Liu A and Xueliang~Li B.
\newblock
  \href{https://www.sciencedirect.com/science/article/pii/S0024379514006971}{Hermitian-adjacency
  matrices and Hermitian energies of mixed graphs}.
\newblock {\em Linear Algebra and its Applications}, 466(466):182--207, 2015.

\bibitem{Bojan2016Hermitian}
Bojan and Mohar.
\newblock
  \href{https://www.sciencedirect.com/science/article/pii/S0024379515006138?via%3Dihub}{Hermitian
  adjacency spectrum and switching equivalence of mixed graphs}.
\newblock {\em Linear Algebra and Its Applications}, 2016.

\bibitem{2016Hermitian}
K.~Guo and B.~Mohar.
\newblock \href{https://arxiv.org/pdf/1505.01321.pdf}{Hermitian Adjacency
  Matrix of Digraphs and Mixed Graphs}.
\newblock {\em Journal of Graph Theory}, 2016.

\bibitem{Krystal2017Digraphs}
Krystal, Guo, Bojan, and Mohar.
\newblock
  \href{https://www.sciencedirect.com/science/article/pii/S0012365X17300286}{Digraphs
  with Hermitian spectral radius below 2 and their cospectrality with paths}.
\newblock {\em Discrete Mathematics}, 2017.

\bibitem{Reff2016Oriented}
Reff and Nathan.
\newblock
  \href{https://www.sciencedirect.com/science/article/pii/S0024379516302154}{Oriented
  Gain Graphs, Line Graphs and Eigenvalues}.
\newblock {\em Linear Algebra and Its Applications}, pages 316--328, 2016.

\end{thebibliography}
\end{document}